\documentclass[12pt]{amsart}

\usepackage{amsmath, amsthm, amscd, amssymb, amsfonts, pstricks,bbm, verbatim,lineno,xcolor,enumitem}
\usepackage[top=3cm, bottom=3cm, left=2.54cm, right=2.54cm]{geometry}

\setlist[enumerate]{label=(\arabic*)}
\setlist{nosep}
\allowdisplaybreaks

\newcommand{\Rb}{{\mathbb R}}

\newcommand{\1}{\mathbf{1}}
\newcommand{\zero}{\mathbf{0}}

\newcommand{\diam}{\mathop{\mathrm{diam}}}
\def\ip<#1,#2>{\langle #1, #2 \rangle}
\newcommand\norm[1]{\Vert #1 \Vert}
\newcommand\mnorm[1]{\bigl\Vert #1 \bigr\Vert}
\newcommand{\st}{\,:\,}

\newtheorem{thm}{Theorem}[section]
\newtheorem*{thm*}{Theorem}
\newtheorem{cor}[thm]{Corollary}
\newtheorem*{cor*}{Corollary}
\newtheorem{lem}[thm]{Lemma}
\newtheorem{prop}[thm]{Proposition}

\newtheorem*{con*}{Conjecture}
\newtheorem*{prob*}{Problem}

\theoremstyle{definition}
\newtheorem{defn}[thm]{Definition}

\theoremstyle{remark}
\newtheorem{rem}[thm]{Remark}
\newtheorem{ex}[thm]{Example}

\title{A problem on distance matrices of subsets of the Hamming cube}
\author{Ian Doust}
\address{School of Mathematics and Statistics, UNSW Sydney NSW 2052, Australia}
\email{i.doust@unsw.edu.au}

\author{Reinhard Wolf}
\address{Institut f{\"u}r Mathematik, Universit{\"a}t Salzburg, Hellbrunnerstrasse 34, A-5020 Salzburg, Austria}
\email{Reinhard.Wolf@sbg.ac.at}

\date{September 2021}

\begin{document}

\begin{abstract}
Let $D$ denote the distance matrix for an $n+1$ point metric space $(X,d)$. In the case that $X$ is an unweighted metric tree, the sum of the entries in $D^{-1}$ is always equal to $2/n$. Such trees can be considered as affinely independent subsets of the Hamming cube $H_n$, and it was conjectured that the value $2/n$ was minimal among all such subsets. In this paper we confirm this conjecture and give a geometric interpretation of our result which applies to any subset of $H_n$.
\end{abstract}

\maketitle

\section{Introduction}
There is a long study of the interaction between properties of finite metric spaces and properties of their distance matrices. The most classical questions in this area concern whether the metric space can be isometrically embedded in a Euclidean space, a problem solved by Schoenberg \cite{Sch}, or else in some other standard normed space. Properties of the metric space are often reflected in linear algebraic properties of the distance matrix involving say the determinant or inverse of the matrix.

To fix some notation, let $(X,d)$ denote a finite metric space with elements $\{x_1,\dots,x_m\}$ (with $m \ge 2$) and let $D = D_X$ denote its distance matrix $(d(x_i,x_j))_{i,j=1}^m$. 
Let $\1$ denote the vector $(1,\dots,1)^T \in \Rb^m$ so that for any $m \times m$ matrix $A$, $\ip<A \1,\1>$ gives the sum of the entries in $A$.

One particular class of spaces for which this relationship has been much studied  are those which are (isometric to) subsets of Hamming cubes $H_n = \{0,1\}^n$ with the Hamming metric $d_1$ (which is the $\ell^1$ metric on $\Rb^n$, restricted to these spaces). (See, for example, \cite{DRSW,GL,GP,GW1,GW2}.) This class includes, for example, all unweighted metric trees.  Much of \cite{DRSW} concerns extending Graham and Pollak's \cite{GP} formula, $\det(D) = (-1)^n n 2^{n-1}$, for the distance matrix of an $n+1$ point unweighted metric tree.

If $X$ is an $n+1$ point unweighted metric tree in $H_n$, then $D$ is invertible. In \cite{DRSW} it was shown that the subsets of $H_n$ for which the distance matrix $D$ is invertible are precisely the ones for which the points form an affinely independent subset of $\Rb^n$. They showed that for an $n+1$ point unweighted metric tree the sum of the entries in $D^{-1}$, that is $\ip<D^{-1} \1,\1>$,  is always equal to $2/n$ and conjectured, based on empirical evidence, that this value was minimal among all affinely independent subsets of $H_n$. The aim of this note is to prove this conjecture and to provide geometric interpretations for the value of this quantity.

A consequence of Theorem~\ref{Main} below will be  the following two results.

\begin{thm}\label{conj-sol}
Suppose that $X$ is an affinely independent subset of the Hamming cube $(H_n,d_1)$ with at least two points. Then
   \begin{equation}\label{Dinv1-1}  \ip<D^{-1} \1,\1>  = \left(\frac{n}{2} - 2 d_2(h,Z_X)^2\right)^{-1} \end{equation}
where $d_2(h,Z_X)$ is the Euclidean distance from the point $h = \bigl(\frac{1}{2},\dots,\frac{1}{2}\bigr)$ to the affine subspace $Z_X \subseteq \Rb^n$ spanned by the elements of $X$.
\end{thm}

\begin{cor}
Suppose that $X$ is an affinely independent subset of the Hamming cube $(H_n,d_1)$ with at least two points.  Then
  \[ \frac{2}{n} \le \ip<D^{-1} \1,\1> \le 2. \]
\end{cor}

\begin{proof}
Let $x$ and $y$ be two distinct points in $X$.
The point $m = (x+y)/2 = (m_1,\dots,m_n)$ must lie in $Z_X$ and so $d_2(h,Z_X) \le \norm{h-m}_2$.
Note that for each $j$, $\bigl|m_j - \frac{1}{2}\bigr|$ is either $0$ or $\frac{1}{2}$, and, as $x$ and $y$ are distinct, the value 0 must occur for at least one value of $j$. Thus $\norm{h-m}_2^2  \le (n-1)/4$ and so
$\frac{1}{2} \le \frac{n}{2} - 2 d_2(h,Z_X)^2 \le \frac{n}{2}$.
Together with (\ref{Dinv1-1}), this then implies the result.
\end{proof}

\begin{rem}\label{Rem1}
Formula (\ref{Dinv1-1}) is somewhat remarkable as the left-hand side only depends on the distances between the points, and not their positions. On the other hand, the
quantity $d_2(h,Z_X)$ depends only on the linear relationship between the points of $X$, and does not appear to depend on their relative distances. This will be illustrated with some examples at the end of the paper.
\end{rem}

Vital to the proof of (\ref{Dinv1-1}) are the facts that the Hamming cube is of $1$-negative type, and that the natural embedding of $H_n$ into $\Rb^n$ is a so-called S-embedding. The definitions of these concepts are given in Section~\ref{S:Neg-type}. Equation~(\ref{Dinv1-1}) is essentially a special case of a formula involving the $M$-constant $M(X)$ of the space $(X,d)$. The link between the $M$-constant and the radius of a particular sphere containing $X$ is due to Nickolas and Wolf \cite[Section 3]{NWQM3} (following earlier work of Alexander and Stolarsky \cite{AS}), and it is this which provides the geometric meaning for many of the quantities considered.
Working with $M(X)$ is advantageous as it is defined even when the matrix $D$ is not invertible, and this will allow us to consider arbitrary subsets of $H_n$.
The relationship between $M(X)$ and $\ip<D^{-1} \1,\1>$ is developed in Section~\ref{S:M-const}, and this provides sufficient information to prove the conjecture in \cite{DRSW}. In the final sections we use the properties of S-embeddings to prove Theorem~\ref{conj-sol} and  to give a geometric interpretation of the value of $\ip<D^{-1} \1,\1>$.

To simplify the statements of the results, we shall assume throughout that all metric spaces considered have at least two elements. (Without this restriction the statements are usually either false or meaningless.)

\section{Negative type and $S$-embeddings}\label{S:Neg-type}

\begin{defn} Suppose that $(X,d)$ is a metric space and that $p \ge 0$. Then
$(X,d)$ is of $p$-negative type if for each finite subset $\{x_1,\dots,x_m\} \subseteq X$ and each set of scalars $\xi_1,\dots,\xi_m \in \Rb$ with $\xi_1 + \dots + \xi_m = 0$,
  \begin{equation}\label{p-neg}
      \sum_{i,j=1}^m d(x_i,x_j)^p \xi_i \xi_j \le 0.
  \end{equation}
\end{defn}

It certain settings, spaces of $1$-negative type are also called quasihypermetric spaces (see for example \cite{NWF}) or spaces of generalized roundness $1$ (see \cite{LTW}).

A space is of strict $p$-negative type if (\ref{p-neg}) holds, with equality only in the trivial case where each $\xi_i$ is zero. (It is worth noting that a distinct but related concept, that of a strictly quasihypermetric space, appears in \cite{NWQM2}. For finite spaces, such as the ones considered in this paper, a space is of strict $1$-negative type if and only if it is strictly quasihypermetric.)
It follows from the results of Wolf \cite{WGap} and S{\'a}nchez \cite{Sa} that a finite metric space of $1$-negative type is of strict $1$-negative type if and only if $D$ is non-singular and $\ip<D^{-1}\1,\1> \ne 0$.

By \cite[Theorem~4.10]{WW} any subset of $\Rb^n$ with the $\ell^1$ metric, and hence any subset of $H_n$, has $1$-negative type. Combining the results of Muragan \cite{Mu}, and Doust, Robertson, Stoneham and Weston \cite{DRSW} (see also Nickolas and Wolf \cite{NWF})  gives the following equivalences.

\begin{thm}\label{HnStrict}
Suppose that $X = \{x_1,\dots,x_m\} \subseteq H_n \subseteq \Rb^n$. Then the following are equivalent.
\begin{enumerate}
    \item $X$ is of strict $1$-negative type.
    \item $X$ is affinely independent (as a subset of $\Rb^n$).
    \item $D$ is non-singular.
\end{enumerate}
\end{thm}

\begin{proof} The equivalence of (1) and (2) was shown in \cite[Theorem~4.3]{Mu}.
The equivalence of (2) and (3) was proven in \cite[Corollary~2.5]{DRSW}.
\end{proof}

Clearly then, $H_n$ is of $1$-negative type, but not of strict $1$-negative type.

A celebrated theorem of Schoenberg \cite{Sch} says that a metric space $(X,d)$ can be isometrically embedded in a Euclidean space if and only if it is of $2$-negative type. This gives the following.

\begin{prop} Let $(X,d)$ be a finite metric space. Then the following are equivalent.
\begin{enumerate}
  \item $(X,d)$ is of $1$-negative type (or quasihypermetric).
  \item $(X,d^{1/2})$ embeds isometrically in a Euclidean space.
\end{enumerate}
\end{prop}

An embedding $\iota: X \to \Rb^n$ which maps $(X,d^{1/2})$ isometrically into $(\Rb^n,\norm{\cdot}_2)$ is called an \textbf{S-embedding}. It is easy to check that the natural inclusion of $H_n$ in $\Rb^n$ is such an embedding, and hence so is the restriction to any subset of $H_n$.

\section{The $M$-constant and maximal measures}\label{S:M-const}

The quantity $\ip<D^{-1} \1,\1>$ is closely related to the $M$-constant of the metric space, which we shall now introduce.  Working with the $M$-constant is in fact usually preferable since it is defined even when the distance matrix $D$ is not invertible. For further background on the $M$-constant we refer the reader to \cite{AS} or \cite{NWQM1}.

Let $(X,d)$ be a compact metric space. For a signed Borel measure $\mu$ on $X$,
let
  \[ I(\mu) = \int_X \int_X d(x,y) \, d\mu(x) d\mu(y) \]
and define $d_\mu: X \to \Rb$ by
  \[ d_\mu(x) = \int_X d(x,y) \, d\mu(y). \]
Let $F_1$ denote the set of measure on $X$ of total mass one. The $M$-constant of
$(X,d)$ is defined to be
  \[ M(X) = \sup_{\mu \in F_1} I(\mu). \]
If $I(\mu) = M(X)$, then we say that $\mu$ is a \textbf{maximal measure}.
It is clear that if $X$ is a metric subspace of $Y$ then $M(X) \le M(Y)$.

Suppose now that $X = \{x_1,\dots,x_m\}$ is a finite metric space.
In this case we shall write $\mu = (\alpha_1,\dots,\alpha_m)$ to denote that $\mu(\{x_i\}) = \alpha_i$, $i=1,\dots,m$.
Then
  \[ I(\mu) = \sum_{i,j=1}^m \alpha_i \alpha_j d(x_i,x_j) = \ip<D \mu,\mu>, \]
and
  \[ d_\mu(x) = \sum_{i=1}^m \alpha_i d(x_i,x)\]
although in most cases we shall retain the integral notation.
We shall identify the measures of total mass one with the hyperplane of vectors whose elements sum to $1$. That is,  $F_1 = \{v \in \Rb^m \st \ip<v,\1> = 1\}$,
and so $M(X) = \sup_{\mu \in F_1} \ip<D \mu,\mu>$.
By considering $\mu = \frac{1}{m}\1$, we have  $M(X) \ge \frac{1}{m^2} \ip<D\1,\1> \ge  \frac{m-1}{m} d_0$, where $d_0$ is the smallest nonzero distance in $X$. In particular $M(X)$ is always strictly positive.

It is less clear that for a general compact metric space $M(X)$ should always be finite, and indeed this need not be the case (see \cite[Theorem 3.1]{NWQM1}). Even if $M(X)$ is finite it may be that there are no maximal measures. Fortunately for subsets of the Hamming cube, these complications do not arise.
Nickolas and Wolf \cite[Theorem 4.7]{NWF} showed that if $X$ is any $m$-point subset of $\Rb^n$ with the $\ell^1$ metric then $M(X) \le \frac{m}{4} \diam(X)$ .

We recall some important properties of these quantities.

\begin{thm}\label{max-measures}
Suppose that $(X,d)$ is a finite metric space of $1$-negative type and that $M(X) < \infty$.
\begin{enumerate}
\item A maximal measure exists.
\item If $\mu$ is a maximal measure, then $d_\mu(x) = M(X)$ for all $x \in X$.
\item If $\mu \in F_1$ and there is a constant $C$ such that $d_\mu(x) = C$ for all $x \in X$, then $\mu$ is maximal and so $M(X) = C$.
\end{enumerate}
\end{thm}

\begin{proof} (1) is \cite[Theorem~4.11]{NWQM2}; (2) and (3) are from \cite[Theorem~3.1]{NWQM2}.
\end{proof}

Theorem~\ref{max-measures} is closely related to the following result.

\begin{thm}\label{M-b}
Suppose that $(X,d)$ is a finite metric space of $1$-negative type with distance matrix $D$. Then there exists $b \in \Rb^m$ such that $Db = \1$ and $\ip<b,\1> \ge 0$. Further
\begin{enumerate}
\item The value of $\ip<b,\1>$ is independent of $b$. That is, if $Db = Db' = \1$ then $\ip<b,\1> = \ip<b',\1>$.
\item $M(X) < \infty$ if and only if $\ip<b,\1> > 0$. In this case $\mu = \frac{1}{\ip<b,\1>} b$ is a maximal measure and
    \[ M(X) = \frac{1}{\ip<b,\1>}. \]
\end{enumerate}
\end{thm}

\begin{proof} The existence of $b$ is shown in \cite[Theorem 4.2]{WEst}. The independence of the value of $\ip<b,\1>$ was noted in \cite[Remark 4.4]{WEst}. Statement (2) is \cite[Theorem 4.8]{WEst}.
\end{proof}

\begin{thm}\label{M-Dinv}
Suppose that $(X,d)$ is a finite metric space of strict $1$-negative type with distance matrix $D$. Then $M(X) < \infty$ and
  \[ M(X) = \frac{1}{\ip<D^{-1}\1,\1>}. \]
\end{thm}

\begin{proof} By Theorem~\ref{HnStrict}, $D$ must be invertible. Let $b = D^{-1}\1$. By \cite[Theorem~4.3]{WEst}, $\ip<b,\1> > 0$ so by Theorem~\ref{M-b}, $M(X) < \infty$ and $M(X) = \bigl(\ip<D^{-1}\1,\1>\bigr)^{-1}$.
\end{proof}

\begin{thm}\label{MHn}
$M(H_n) = \frac{n}{2}$.
\end{thm}

\begin{proof}
Due to the symmetry of the Hamming cube, the sum of the distances from any given point is independent of the point. Simple analysis shows that this sum is
  \[ \beta = \sum_{k=0}^n k \binom{n}{k} = n 2^{n-1} \]
and so $D \1 = \beta \1$.
Let $b = \frac{1}{\beta} \1$, so $Db = \1$ and
  \[ \ip<b,\1> = \frac{1}{n 2^{n-1}} \ip<\1,\1> = \frac{2^n}{n 2^{n-1}} = \frac{2}{n}. \]
By Theorem~\ref{M-b} then
  \[ M(H_n) = \frac{n}{2}. \]
\end{proof}

Combining the above results gives a positive answer to the conjecture in \cite{DRSW}.

\begin{thm}\label{conj-ans}
Let $X$ be a subset of $H_n$. Then
\begin{enumerate}
    \item $M(X) \le \frac{n}{2}$.
    \item If $X$ is affinely independent, then $\ip<D^{-1} \1,\1> = \frac{1}{M(X)}$ and hence $\ip<D^{-1} \1,\1> \ge \frac{2}{n}$.
\end{enumerate}
\end{thm}

\begin{proof}
(1) As noted earlier if $X \subseteq H_n$ then $M(X) \le M(H_n)$, so the result follows immediately from Theorem~\ref{MHn}.

(2) This follows immediately from  Theorem~\ref{HnStrict},  Theorem~\ref{M-Dinv} and (1).
\end{proof}

The remainder of the paper is devoted to investigating the geometric interpretation of $M(X)$ in the context of subsets of the Hamming cube.

\section{S-embeddings and spheres}\label{S-S-embed}

There is a close connection between S-embeddings onto spheres and maximal measures. We begin with two lemmas.

\begin{lem}\label{Lemma1} Suppose that $u_1,\dots,u_m \in \Rb^n$ and that $\alpha_1,\dots, \alpha_m \in \Rb$ satisfy $\sum_{i=1}^m \alpha_i = 1$. Then, for all $u \in \Rb^n$,
  \[ \sum_{i=1}^m \alpha_i \norm{u_i-u}_2^2
      = \mnorm{\sum_{i=1}^m \alpha_i u_i-u}_2^2
         + \frac{1}{2} \sum_{i,j=1}^m \alpha_i \alpha_j \norm{u_i-u_j}_2^2. \]
\end{lem}

\begin{proof} With the notation of the lemma
  \begin{align*}
  \sum_{i=1}^m \alpha_i \norm{u_i-u}_2^2
   &= \sum_{i=1}^m \alpha_i \norm{u_i}_2^2 + \sum_{i=1}^m \alpha_i \norm{u}_2^2
               - 2 \sum_{i=1}^m \alpha_i \ip<u_i,u> \\
   &= \norm{u}_2^2 - 2  \ip<\sum_{i=1}^m \alpha_i  u_i,u>
        + \mnorm{\sum_{i=1}^m \alpha_i u_i}_2^2 - \mnorm{\sum_{i=1}^m \alpha_i u_i}_2^2
          + \sum_{i=1}^m \alpha_i \norm{u_i}_2^2 \\
   &= \mnorm{\sum_{i=1}^m \alpha_i u_i - u}_2^2 - \mnorm{\sum_{i=1}^m \alpha_i u_i}_2^2
          + \sum_{i=1}^m \alpha_i \norm{u_i}_2^2 \\
   &= \mnorm{\sum_{i=1}^m \alpha_i u_i - u}_2^2 - \sum_{i,j=1}^m \alpha_i \alpha_j \ip<u_i,u_j>
          + \sum_{i=1}^m \alpha_i \norm{u_i}_2^2 \\
   &= \mnorm{\sum_{i=1}^m \alpha_i u_i - u}_2^2 + \sum_{i=1}^m \alpha_i \norm{u_i}_2^2
       - \frac{1}{2} \sum_{i,j=1}^m \alpha_i \alpha_j
             \left(\norm{u_i}_2^2 + \norm{u_j}_2^2 - \norm{u_i-u_j}_2^2 \right) \\
   &= \mnorm{\sum_{i=1}^m \alpha_i u_i - u}_2^2 + \sum_{i=1}^m \alpha_i \norm{u_i}_2^2 \\
   & \qquad - \frac{1}{2} \sum_{j=1}^m \alpha_j \sum_{i=1}^m \alpha_i \norm{u_i}_2^2
           - \frac{1}{2} \sum_{i=1}^m \alpha_i \sum_{j=1}^m \alpha_j \norm{u_j}_2^2
           + \frac{1}{2} \sum_{i,j=1}^m \alpha_i \alpha_j \norm{u_i-u_j}_2^2 \\
   &= \mnorm{\sum_{i=1}^m \alpha_i u_i-u}_2^2
         + \frac{1}{2} \sum_{i,j=1}^m \alpha_i \alpha_j \norm{u_i-u_j}_2^2.
  \end{align*}
\end{proof}

\begin{lem}\label{Lemma2}
Let $(X,d)$, $X = \{x_1,\dots,x_m\}$ be a finite metric space of $1$-negative type, and let $\iota: (X,d^{1/2}) \to (\Rb^n,\norm{\cdot}_2)$ be an S-embedding of $X$. Suppose that $\mu = (\alpha_1,\dots,\alpha_m) \in F_1$. Then, for all $x \in X$,
  \[ d_\mu(x) = \mnorm{\sum_{i=1}^m \alpha_i \iota(x_i)-\iota(x)}_2^2
                   + \frac{I(\mu)}{2}. \]
\end{lem}

\begin{proof} Using Lemma~\ref{Lemma1}
\begin{align*}
  d_\mu(x) & = \sum_{i=1}^m \alpha_i \norm{\iota(x_i) - \iota(x)}_2^2 \\
     &= \mnorm{\sum_{i=1}^m \alpha_i \iota(x_i)-\iota(x)}_2^2
         + \frac{1}{2} \sum_{i,j=1}^m \alpha_i \alpha_j \norm{\iota(x_i)-\iota(x_j)}_2^2 \\
     &= \mnorm{\sum_{i=1}^m \alpha_i \iota(x_i)-\iota(x)}_2^2
         + \frac{1}{2} \sum_{i,j=1}^m \alpha_i \alpha_j d(x_i,x_j) \\
     &= \mnorm{\sum_{i=1}^m \alpha_i \iota(x_i)-\iota(x)}_2^2
         + \frac{I(\mu)}{2}.
\end{align*}

\end{proof}


 The content of the following result can be found in Theorem~3.2 of \cite{NWQM3}. We include a short proof for completeness.

\begin{thm}\label{Theorem1} Let $(X,d)$, $X = \{x_1,\dots,x_m\}$ be a finite metric space of $1$-negative type, and let $\iota: (X,d^{1/2}) \to (\Rb^n,\norm{\cdot}_2)$ be an S-embedding of $X$. Then
\begin{enumerate}
  \item If $\mu = (\alpha_1,\dots,\alpha_m)$ is a maximal measure on $X$, then $\iota(X)$ lies on a sphere in $\Rb^n$ with centre $\sum_{i=1}^m \alpha_i \iota(x_i)$ and radius
        \[ r = \sqrt{\frac{M(X)}{2}}. \]
  \item Suppose that $\iota(X)$ lies on a sphere of radius $r$ with centre $c$ which lies inside the affine hull of $\{\iota(x_1),\dots,\iota(x_m)\}$, say $c = \sum_{i=1}^m \beta_i \iota(x_i)$ with $\sum_{i=1}^m \beta_i = 1$. Then $\mu = (\beta_1,\dots,\beta_m)$ is a maximal measure on $X$ and
        \[ r = \sqrt{\frac{M(X)}{2}}. \]
\end{enumerate}
\end{thm}

\begin{proof}
(1) Let $\mu = (\alpha_1,\dots,\alpha_m) \in F_1$ be a maximal measure on $X$.  Fix then $x \in X$.
 By Theorem \ref{max-measures}  and Lemma \ref{Lemma2}
  \[ M(X) = d_\mu(x) = \mnorm{\sum_{i=1}^m \alpha_i \iota(x_i) - \iota(x)}_2^2
            + \frac{I(\mu)}{2}. \]
Since $\mu$ is maximal $I(\mu) = M(X)$ and hence
 \[ \mnorm{\sum_{i=1}^m \alpha_i \iota(x_i) - \iota(x)}_2^2
          = \frac{M(X)}{2} \]
which proves (1).

(2) Let $c$ be as in the statement of the theorem and suppose that $x \in X$, so that $\norm{c - \iota(x)}_2^2 = r^2$. Let $\mu = (\beta_1,\dots,\beta_m)$. By Lemma~\ref{Lemma2}
  \[ d_\mu(x) = \norm{c - \iota(x)}_2^2 + \frac{I(\mu)}{2} = r^2 + \frac{I(\mu)}{2}. \]
Since $d_\mu(x)$ is independent of $x$, from Theorem~\ref{max-measures} we can conclude that $\mu$ is maximal on $X$ and that
  \[ M(X) = r^2 + \frac{I(\mu)}{2}  \]
and hence that
  \[ r^2 = \frac{M(X)}{2}. \]
\end{proof}

Suppose that $B = \{v_1,\dots,v_k\}$ is a basis for a subspace $Z \subseteq \Rb^n$.
Then there is a unique point $c \in Z$ which is equidistant from all the elements of $B$ and the origin.
Indeed a small calculation shows that if $A$ is the $n \times k$ matrix whose $i$th column is $v_i$, then 
  \begin{equation}\label{basis-sphere}
c = \sum_{i=1}^k \gamma_i v_i
  \qquad \text{where} \qquad
  \begin{pmatrix} \gamma_1 \\ \vdots \\ \gamma_k \end{pmatrix}
   = \frac{1}{2} (A^TA)^{-1}
           \begin{pmatrix} \norm{v_1}_2^2 \\ \vdots \\ \norm{v_k}_2^2 \end{pmatrix}.
   \end{equation}
This implies the following.

\begin{lem}\label{unique-S} Suppose that $X = \{x_1,\dots,x_m\}$ is a finite subset of $\Rb^n$ and let $Z_X$ be the smallest affine subspace of $\Rb^n$ containing $X$. Then there is at most one sphere in $\Rb^n$ whose centre lies in $Z_X$ and which contains the points of $X$.
\end{lem}

%
%

\section{The $M$-constant for subsets of the Hamming cube}

Suppose that $X = \{x_1,\dots,x_m\}  \subseteq H_n$.
Let $Z_X$ denote the smallest affine subspace of $\Rb^n$ containing the points $\{x_1,\dots,x_m\}$. We shall use $d_2(x,Z_X)$ to denote the Euclidean distance from a point $x \in \Rb^n$ to an affine subspace $Z_X$. Let $h = \bigl(\frac{1}{2},\dots,\frac{1}{2}\bigr) \in \Rb^n$.

\begin{thm}\label{Main}
Suppose that $X = \{x_1,\dots,x_m\} \subseteq H_n$. Then
  \begin{equation}\label{main-ident}
  M(X)  = \frac{n}{2} - 2 d_2(h,Z_X)^2.
  \end{equation}
\end{thm}

\begin{proof}
Let $\iota: (H_n,d_1^{1/2}) \to (\Rb^n,\norm{\cdot}_2)$ be the natural inclusion map of $H_n$ in $\Rb^n$. As noted in Section~\ref{S:Neg-type}, this map is necessarily an S-embedding.

Since $\norm{x-h}_2^2 = \frac{n}{4}$ for all $x \in H_n$, we have that $X = \iota(X)$ lies on the sphere $S \subseteq \Rb^n$ of radius $r = \sqrt{n}/2$ centred at $h$. Let $P$ be the orthogonal projection from $\Rb^n$ onto $Z_X$ and let $c_X = Ph$. If $u \in \iota(X) \subseteq Z_X$ then, by Pythagoras,
  \[ \norm{c_X-u}_2^2 = \norm{h - u}_2^2 - \norm{h - c_X}_2^2
     = \frac{n}{4} - d_2(h,Z_X)^2. \]
That is, all points in $\iota(X)$ lie on a sphere $S_X$ with centre $c_X \in Z_X$ and radius $r$ with $r^2 = \frac{n}{4} - d_2(h,Z_X)^2$. (Note that by Lemma~\ref{unique-S}, there is only one such sphere.)

But by Theorem~\ref{Theorem1}(2), the radius of such this sphere must also satisfy
  \[ r^2 = \frac{n}{4} - d_2(h,Z_X)^2 = \frac{M(X)}{2}\]
which gives the result.
\end{proof}

Equation (\ref{main-ident})  immediately gives the following characterization of when the maximum value of $M(X)$ is achieved.

\begin{cor}
 $M(X)$ achieves its maximum value of $\frac{n}{2}$ if and only if $h$ lies in $Z_X$.
\end{cor}

Combining Theorem~\ref{Main} and Theorem~\ref{conj-ans} gives Theorem~\ref{conj-sol} stated in the introduction.

Following the proof of \cite[Theorem~3.2]{NWQM3}, an alternative
but less geometrically illuminating verification of Theorem~\ref{Main} can be given by noting that for  $\mu = (\alpha_1,\dots,\alpha_m) \in F_1$, and $\{x_i\}_{i=1}^m \subseteq H_n$,
  \begin{align*}
  I(\mu)
   &=\sum_{i,j = 1}^m \alpha_i \alpha_j \norm{x_i - x_j}_1
   = \sum_{i,j = 1}^m \alpha_i \alpha_j \norm{x_i - x_j}_2^2 \\
   &= \sum_{i,j = 1}^m \alpha_i \alpha_j \norm{(x_i-h) - (x_j-h)}_2^2 \\
   &= \sum_{i,j = 1}^m \alpha_i \alpha_j \norm{(x_i-h)}_2^2
       + \sum_{i,j = 1}^m \alpha_i \alpha_j \norm{(x_j-h)}_2^2
       - 2 \sum_{i,j = 1}^m \alpha_i \alpha_j \ip<x_i-h,x_j-h> \\
   &= \frac{n}{2} - 2 \mnorm{\sum_{i = 1}^m \alpha_i (x_i - h)}_2^2 \\
     & = \frac{n}{2} - 2 \mnorm{\sum_{i = 1}^m \alpha_i x_i - h}_2^2.
  \end{align*}
Maximizing $I(\mu)$ then gives the result.

One consequence of Theorem~\ref{Main} is that the value of $M(X)$ for $X \subseteq H_n$ is determined by the $M$-constant of any maximal affinely independent subset $Y$ of $X$. (Since, by Theorem~\ref{HnStrict}, a maximal affinely independent subset of $X$ is also a maximal subset of strict $1$-negative type, this can also be deduced from Theorem~2.7 of \cite{NWF}.)
Such a set $Y$ may be much smaller than $X$, and furthermore the value of $M(Y)$ may be calculated algorithmically rather than by an optimization process. Finding a suitable affinely independent subset can be easily done using Gaussian elimination. The distance matrix for $Y$ is then invertible, and Theorem~\ref{conj-ans} implies that $M(X) = M(Y) = (\ip<D_Y^{-1}\1,\1>)^{-1}$.

Alternatively, if $Y = \{y_0,\dots,y_m\}$ and $v_i = y_i - y_0$, $i=1,\dots, m$, then one may use (\ref{basis-sphere}) to compute the centre $c$ of the sphere in $\text{span}(v_1,\dots,v_m)$ containing the points $\zero,v_1,\dots,v_m$. Then $M(X) = M(Y) = 2 \norm{c}_2^2$. In the case that $X$ is affinely independent, one may therefore use Theorem~\ref{conj-sol}, the proof of Theorem~\ref{Main}, and Pythagoras
to see that $\ip<D^{-1}\1,\1>$ is equal to $(2r^2)^{-1}$ where $r$ is the radius of the smallest sphere containing all the points in $X$.

We finish with two small examples which illustrate Remark~\ref{Rem1} concerning the lack of an obvious relationship between the distance matrix $D$ and the subspace $Z_X$ which appear on the two sides of Equation~(\ref{Dinv1-1}).

\begin{ex}
Let $X_1 = \{(0,0,0),(1,1,1)\}$ and let $X_2 = X_1 \cup \{(1,0,0)\}$. In this case $Z_{X_1}$ is different to $Z_{X_2}$. The point $h$ lies in both subspaces and hence $M(X_1) = M(X_2) = \frac{3}{2}$. Of course the distance matrices are quite different with
  \[ D_{X_1}^{-1} = \begin{pmatrix}
         0   & \tfrac{1}{3} \\
     \tfrac{1}{3} &  0
       \end{pmatrix},
       \qquad
     D_{X_2}^{-1} = \begin{pmatrix}
      -\tfrac{1}{3}  & \tfrac{1}{6}   & \tfrac{1}{2} \\
       \tfrac{1}{6}  & -\tfrac{1}{12} & \tfrac{1}{4} \\
       \tfrac{1}{2}  & \tfrac{1}{4}   & -\tfrac{3}{4}
       \end{pmatrix},
       \]
but the sum of the entries of each matrix inverse is $\frac{2}{3}$.
\end{ex}

\begin{ex}
Let $X_1 = \{(0,0,0),(1,0,0),(0,1,0)\}$ and let $X_2 = X_1 \cup \{(1,1,0)\}$. Then $Z_{X_1} = Z_{X_2}$ and so by Theorem~\ref{Main} we must have $M(X_1) = M(X_2)$. Here $X_1$ is affinely independent and $\ip<D_{X_1}^{-1} \1,\1> = M(X_1)^{-1} = 1$. However $X_2$ is not affinely independent and $D_{X_2}$ is not invertible. (Using Lagrange multipliers, one can confirm, directly from the definition, that $M(X_2) = 1$. Alternatively, one may use Theorem~\ref{Theorem1} since $X_2$ certainly lies in a sphere with centre in $Z_{X_2}$ and radius $1/\sqrt{2}$.)
\end{ex}

\end{document}